\documentclass[12pt]{amsart}

\usepackage[margin=1in]{geometry}          
\usepackage{amsmath, amscd, amssymb}     
\usepackage{amsfonts}    
\usepackage{amsthm}
\usepackage[all]{xy}               

\swapnumbers
\numberwithin{equation}{section}

\def\m{\medskip}

\newtheorem{thm}{Theorem}[section]

\newtheorem{prop}[thm]{Proposition}
\newtheorem{cor}[thm]{Corollary}

\newtheorem{rem}[thm]{Remark}
\newtheorem{defi}[thm]{Definition}

\newcommand\theoref{Theorem~\ref}

\newcommand\propref{Proposition~\ref}

\newcommand\remref{Remark~\ref}

\def\cat{\operatorname{cat}}
\def\pt{\operatorname{pt}}

\begin{document}

\title[LS Category of connected sums]{On Lusternik-Schnirelmann Category of Connected Sums}

\author{Robert J. Newton } 
\address{Department of Mathematics, 
University of Florida,
Gainesville, Florida 32608}

\begin{abstract}
  In this paper we estimate the Lusternik-Schnirelmann category of the connected sum of two manifolds through their categories. We achieve a more general result regarding the category of a quotient space $X/A$ where $A$ is a suitable subspace of $X$.
\end{abstract}

\maketitle

\section{Introduction}

\begin{defi}{\rm The Lusternik-Schnirelmann category (LS category) of a space $X$ is the smallest nonnegative integer $n$ such that there exists \{$A_0, A_1, ... , A_n$\}, an open cover of $X$ with each $A_i$ contractible in X. This is denoted by $n= \cat X$.}
\end{defi}

Following this definition, spaces with LS category 0 are contractible.

\m The goal of the paper is to prove the inequality\\
\begin{equation}\label{eq:main}
\max\{\cat  M, \cat N\} -1 \leq \cat(M\# N) \leq \max\{\cat M, \cat N\},
\end{equation} where $M$ and $N$ are closed manifolds.

\m To prove the inequality, we consider a more general problem about the relation of $\cat X$ and $\cat(X/A)$. This problem is indeed more general: in fact, put $X=M\#N$ and $A$ be an $(n-1)$ sphere that separates $M$ and $N$ (with removed discs). Then $X/A=M\vee N$ and  $\cat (M\vee N)=\max\{\cat M, \cat N\}$.

\m All spaces are assumed to be CW spaces.

\m {\bf Acknowledgments:} {A special thanks my advisor, Dr. Yuli Rudyak, and to Dr. Alexander Dranishnikov for their continued support and patience.}

\section{Preliminaries}

\begin{defi}\rm  For a path-connected space $X$ with basepoint $x_0$, we define $PX$ to be the set of all continuous functions $\gamma : I\rightarrow X$ satisfying $\gamma (0)=x_0$ topologized by the compact-open topology.
\end{defi}

We then define $p:PX\rightarrow X$ given by $p(\gamma)=\gamma (1)$, a fibration with base space $X$ and fiber $\Omega(X)$, the loop space of $X$.\\

Given $f:Y\rightarrow X$ and $g:Z\rightarrow X$ we can define $Y*_X Z = \{(y, z, t)\in Y*Z | f(y) = g(z)\}$ and $(f*_X g):Y*_X Z\to X$ by $(f*_X g)(y, z, t)= f(y)$.
\\

From this, we define $P_n X$ to be the fiberwise join of $n$ copies of $PX$ over $p:PX\rightarrow X$ defined above and denote the fiberwise map as 
\begin{equation}\label{p_n}
p_n^X:P_n X\rightarrow X.
\end{equation}
 Note that the (homotopy) fiber of $p_n^X$ is $\Omega(X)^{*n}$, the n-fold join of $\Omega(X)$.

\m We need the following theorem of Schwarz, see~\cite{CLOT, Sv}.

\begin{thm}\label{t:Sv} The inequality  $\cat(X) \leq n$ holds iff there exists a section $s:X\rightarrow P_{n+1}X$ to $p:P_{n+1}X\rightarrow X$.
\end{thm}

\begin{rem}\label{r:help}\rm These claims are well-known, we list them here for reference.

\item(1) $\cat(X\vee Y) = \max\{\cat X, \cat Y\}$;
\item(2) $\cat(X\cup Y)\leq \cat(X) + \cat(Y) + 1$;
\item(3)  $\cat(X/A)-1\leq \cat X $. This follows from the fact that $X/A$ has homotopy type of $X\cup CA$, the union of $X$ with the cone over $A$, and item (2).

\end{rem}

It should be noted that Berstein and Hilton explored the changes in category of a space after attaching a cone in \cite{BH} following Hilton's exploration of what's now known as the Hilton-Hopf invariant in \cite{Hi}.

\section{Main results}

\begin{prop}\label{p:smallcat} The inequality \eqref{eq:main} holds whenever $\max\{\cat  M, \cat N\}\leq 2$
\end{prop}

\begin{proof}If $\cat M=1=\cat N$ then $M$ and $N$ are homotopy spheres, and so $M\# N$ is. Conversely, if $\cat M\# N=1$ then $M$ and $N$ must be homotopy spheres.  Thus, we proved that $\cat(M\# N)=\max\{\cat M, \cat N\}$ if $\max\{\cat M, \cat N\}\le 2$.
\end{proof}

\begin{thm}\label{t:main} Suppose $X$ is an $n$-dimensional space with m-connected subspace $A$, with $3 \leq \cat(X/A)  \leq k$, and $k+m-1\geq$n. Then $\cat X \leq k$.
\end{thm}

\begin{proof}  For sake of simplicity, put $p=p^{X/A}_{k+1}$ and $p'=p_{k+1}^X$, cf.~\eqref{p_n}.  As $\cat(X/A) \leq k$, and by \theoref{t:Sv}, there exists the following section $s$ with $ps=1_{X/A}$.

\hspace{2in}\xymatrix{ P_{k+1}(X/A) \ar[d]_p	\\  
X/A 	 \ar@/_/@{.>}[u]_s}

Now we consider the collapsing  map $q:X\rightarrow X/A$, and get the fiber-pullback diagram.
\begin{equation}
\CD
E @>f>> P_{k+1}(X/A) @= P_{k+1}(X/A)\\
@VVV @VpVV @AAsA\\
X @>>> X/A @= X/A
\endCD
\end{equation}

\m

Now consider $P_{k+1}X$.  We already have $p':P_{k+1}X\rightarrow X$, and the collapsing map $q:X\rightarrow X/A$ induces a map $q':P_{k+1}X\rightarrow P_{k+1}(X/A)$. Since $pq'=gp'$ and the square is the pull-back diagram, we get a map $h: P_{k+1}\to X$  such that  the following diagram commutes.
\\

\hspace{1.5in}\xymatrix{ P_{k+1}X \ar@/_/[ddr]_{p'} \ar@/^/[drr]^{q'} \ar@{.>}[dr]^h
\\ & E \ar[d] \ar[r]^f & P_{k+1}(X/A) \ar[d]_p	
\\ & X \ar[r]^q \ar@/_/@{.>}[u]_{s'}	& X/A 	 \ar@/_/@{.>}[u]_s}
\\

\m

Recall that our goal is to prove $\cat X \le k$. Because of Schwarz's \theoref{t:Sv}, it suffices to construct a section of $p'$. To do this, it suffices in turn to construct a section of the map $h: P_{k+1}(X)\to E$. Moreover, since $\dim X=n$, it suffices to construct a section of $h$ over the $n$-skeleton $E^{(n)}$ of $E$, i.e., to construct a map $\phi: E^{(n)}\to P_{k+1}(X)$ with $h\phi=1_E$.
\\

\m By homotopy excision~\cite[Prop. 4.28]{Ha},  and because $A$ is $m$-connected, the quotient map $q:X\to X/A$ induces isomorphisms  $q_*:\pi_n(X)\rightarrow \pi_n(X/A)$  for $n\leq m$ and epimorphism for $n=m+1$. So,  $\pi_n(\Omega X)\rightarrow \pi_n(\Omega(X/A))$ is an isomorphism for $n \leq (m-1)$ and epimorphism for $n=m$. Therefore $(\Omega X)^{*(k+1)}\rightarrow (\Omega(X/A))^{*(k+1)}$ is an isomorphism for $n$ $\leq$ $m+k$ because of \cite[Prop. 5.7]{DKR}.

\m
The long exact sequence of homotopy groups for a fibration yields the following 
commutative diagram

\hspace{1in}\xymatrix{ ...   \ar[r] & \pi_i((\Omega X)^{*(k+1)}) \ar[r] \ar[d]^\cong &  \pi_i(P_{k+1}X) \ar[r] \ar@{.>}[d]^{h_*} & \pi_i(X) \ar[r] \ar[d]^\cong & ...\\
 ...   \ar[r] & \pi_i((\Omega X)^{*(k+1)}) \ar[r] &  \pi_i(E) \ar[r] & \pi_i(X) \ar[r] & ...}
\\
\\
\\
By the 5-lemma, the map $h_*$ is an isomorphism for $i \leq (m+k-1)$ and epimorphism for $n=m+k$. So by Whitehead's theorem, there exists a map $\phi:E^{(n)}\rightarrow P_{k+1}X$. Now,  the composition $(\phi \circ s')$ is a section to $p':P_{k+1}\to X$. Thus $\cat X \leq k$.
\end{proof}

\m
Combining this with the previous \remref{r:help} gives the following inequality:
$$
\cat(X/A)-1 \leq \cat(X) \leq \cat(X/A)
$$

under the dimension-connectivity conditions from \theoref{t:main}.

\m Consider the case where $X = M\#N$, the connected sum of $n$-dimensional manifolds $M$ and $N$, and $A= S^{n-1}$ is the separating sphere between $M$ and $N$. Then $X/A = M\vee N$, and $\cat (X/A)=\max\{\cat M, \cat N\}$. We have $A$ is $(n-2)$-connected, and we can assume $\cat M, \cat N \geq  3$ because of \propref{p:smallcat}. Then $\cat (X/A) \geq  3$, and so as $(n-2) + 3 - 1 \geq n$, we are in the case of \theoref{t:main} and get the following corollary.

\begin{cor}\label{c:maincor} There is a double inequality
\begin{equation*}
\max\{\cat  M, \cat N\} -1 \leq \cat(M\# N) \leq \max\{\cat M, \cat N\}.
\end{equation*}
\end{cor}

\begin{proof}  Consider the case where $X = M\#N$, the connected sum of $n$-dimensional manifolds $M$ and $N$, and $A= S^{n-1}$ is the separating sphere between $M$ and $N$. Then $X/A = M\vee N$, and $\cat (X/A)=\max\{\cat M, \cat N\}$. We have $A$ is $(n-2)$-connected, and we can assume $\cat M, \cat N \geq  2$ because of \propref{p:smallcat}. Then $\cat (X/A) \geq  3$, and so as $(n-2) + 3 - 1 \geq n$, we are in the case of \theoref{t:main} and get the corollary.
\end{proof}

\begin{rem}\rm In \cite{H}, an upper bound is given for the LS category of a double mapping cylinder. If we consider  the connected sum of $n$-manifolds $M$ and $N$ as such a double mapping cylinder, then the following inequality is obtained:
\begin{equation}\label{eq:hardie}
\cat M\#N\leq\min\{1 + \cat M' + \cat N', 1 + \max\{\cat M', \cat N'\}\}.
\end{equation}
Here $M'$ and $N'$ are $M \smallsetminus \pt$ and $N\smallsetminus \pt$, respectively. 

Rivadeneyra proved in \cite{R} that category of a manifold without bondary does not increase when a point is removed. If the categories of $M$ and $N$ do decrease by one when a point is removed, then \eqref{eq:hardie} has already established the main result here. However, in \cite{LSV} a closed manifold is constructed so that the category remains unchanged after the deletion of a point, and \theoref{t:main} gives an improvement of the category estimate for such a case.
\end{rem}

\begin{rem}\rm It is unknown if there is an example of two manifolds $M$ and $N$ such that  $\cat M\#N = \max\{\cat  M, \cat N\} -1$. 
\end{rem}

\

\section{Connected sum and Toomer invariant}

\begin{defi}\rm The Toomer invariant of $X$ $ e(X)$ is the least integer $k$ for which the map $p_{n}^{*}:H^*(X)\to H^*(P_n(X))$ is injective, see \cite{CLOT}. It follows that $e(X) \le 
\cat X$.
\end{defi}

\begin{prop}\label{Toom}For closed and oriented manifolds $M$ and $N$, $e(M\#N)\geq \max\{e(M),e(N)\}.$
\end{prop}

\begin{proof}
Consider $f:M\#N\to M$ the collapsing map onto $M$. Then we have the following diagram.

\hspace{1.5in}\xymatrix{ & H^*(P_nM)  \ar[r] & H^*(P_{n}(M\#N))
\\ & H^*(M)  \ar[u] \ar@{>->}[r]^{f^*}	& H^*(M\#N) \ar[u]	 }
\\

This map has degree 1, and so $f^*:H^*(M)\to H^*(M\#N)$ is injective \cite[Theorem V, 2.13]{Rudbook}.  Also suppose $p_n^*:H^*(M\#N)\to H^*(P_n(M\#N)$ is injective. Consider $u\in H^*(M)$. As $f^*$ and $p_n^*$ are injective, $p_n^*(u)\in H^*(P_nM)$ is nonzero, and so $p_n^*:H^*(M)\to H^*(P_nM)$ is injective, and similarly for $N$. And so $e(M\#N)\geq \max\{e(M),e(N)\}.$
\end{proof}

\begin{prop}For closed, oriented manifolds $M$ and $N$, if $\cat M = e(M)$ and $\cat N = e(N)$, then $\cat (M\#N)= \max\{ \cat M, \cat N\}$. 
\end{prop}

\begin{proof}Combining the assumptions $e(M)=\cat M$ and $e(N)=\cat N$ with the inequality $\max \{e(M), e(N)\} \leq e(M\#N) \leq \cat(M\#N) \leq \max\{ \cat M, \cat N\}$, we have the claim.
\end{proof}

Rudyak asked if the existence of a map $f:M\to N$, of degree 1, implies the inequality $\cat M \geq \cat N$ \cite{Rudweight},\,\cite[Open problem 2.48]{CLOT}. While not achieving the full result, he was able to prove some partial results.  In particular it follows from the same injective property of $f^*$ (4.2) that $e(M)\geq e(N)$, when such a map exists \cite{Rudweight}.

\begin{rem}\rm We know $e(M \times S^n) \geq e(M)+1$, and there exist examples where $\cat(M\times S^n)=\cat M$ for suitable $M$ and $n$, \cite{Iw98}, \cite{Iw02}.
\end{rem}

\section{Rationalizations}
Here we assume $X$ to be simply connected and denote by $X_\mathbb{Q}$ the rationalization of $X$, see \cite{Fe},\cite{Sull}. We define $e_\mathbb{Q}(X)$ to be the least integer $n$ such that the $n$th fibration $P_nX\to X$ induces an injection in cohomology with coefficients in $\mathbb{Q}$. For $X$ simply connected and of finite type, we have that $e_\mathbb{Q} (X) = e(X_\mathbb{Q})$, \cite{CLOT}.
\\

\begin{prop} For simply connected, CW spaces $X$ and $Y$, $(X \vee Y)_\mathbb{Q} \cong X_\mathbb{Q} \vee Y_\mathbb{Q}$.
\end{prop}
\begin{proof} In the following diagram, the map $l$ is the localization map of $X\vee Y$, and $k$ is given by the wedge of localization maps on $X$ and $Y$. The map $j$ exists by the universal property of $(X\vee Y)_\mathbb{Q}$, and induces isomorphisms in homology. Hence $X_\mathbb{Q}\vee Y_\mathbb{Q} \cong (X\vee Y)_\mathbb{Q}$

\hspace*{1.5in}\xymatrix{ & X\vee Y  \ar [dr]_k \ar[r]^l & (X\vee Y)_\mathbb{Q} \ar@{.>}[d]^j
\\ & 	& X_\mathbb{Q}\vee Y_\mathbb{Q}}

\end{proof}

In \cite{FHL} it is shown that for a closed, simply connected manifold $M$, $e(M) = e_\mathbb{Q}(M)=\cat (M_\mathbb{Q})$, and hence $\cat M_\mathbb{Q}\leq \cat M$.

\begin{prop}\rm For $M$ and $N$, closed and simply connected manifolds, $\cat (M\#N)_\mathbb{Q} = \max \{ \cat M_\mathbb{Q}, \cat N_\mathbb{Q} \}$.
\end{prop}
\begin{proof}
As $M$ and $N$ are closed and simply connected, $M\#N$ is closed and simply connected, and $e_\mathbb{Q}(M\#N) = \cat (M\#N)_\mathbb{Q}$. Combining (3.3) and (4.2) establishes on the left hand side, 
$$\max \{\cat M_\mathbb{Q}, \cat N_\mathbb{Q}\} =\max\{e_\mathbb{Q}(M), e_\mathbb{Q}(N)\}\leq e_\mathbb{Q}(M\#N) = \cat(M\#N)_\mathbb{Q}.$$
While on the right hand side we have, $$ \cat(M\#N)_\mathbb{Q} = \cat_\mathbb{Q}(M\#N)\leq \max (\cat_\mathbb{Q}M, \cat_\mathbb{Q}N)= \max \{\cat M_\mathbb{Q}, \cat N_\mathbb{Q}\},$$ where the middle inequality comes from (3.3).
\end{proof}

Returning to Rudyak's question on a possible relation between degree and category, we can settle it in the rational context. 
\begin{prop}\rm For closed and simply connected $m$-manifolds $M$ and $N$ with $f:M\to N$ of nonzero degree, we have $\cat M_\mathbb{Q} \geq \cat N_\mathbb{Q}$.
\end{prop}

\begin{proof} It suffices to show $e_\mathbb{Q}(M) \geq e_\mathbb{Q}(N)$. That is, suppose $p^*:H^*(M;\mathbb{Q})\to H^*(P_n(M)\mathbb{Q})$ in the following diagram is injective. 
\\

\hspace{1.5in}\xymatrix{ & H^*(P_n(M);\mathbb{Q})   & H^*(P_{n}(N);\mathbb{Q}) \ar[l]
\\ & H^*(M;\mathbb{Q})  \ar[u]^{p^*} \ &  H^*(N;\mathbb{Q}) \ar[u]_{p^*} \ar[l]_{f^*}	}

By \cite[V.2.13]{Rudbook}, the map $f^*$ is injective. Since $p^*$ and $f^*$ are injective, the composition $p^*\circ f^*$ is injective, and $p^*:H^*(N;\mathbb{Q})\rightarrow H^*(P_n(N);\mathbb{Q})$ is injective. Thus $e_\mathbb{Q}(N)\leq n$.

\end{proof}


\begin{thebibliography}{99}

\bibitem{BH} I. Berstein,  and P. Hilton. Category and generalized Hopf invariants.
{\em Illinois J. Math.} \textbf{4} (1960), 437-451. 

\bibitem{CLOT} O. Cornea, G. Lupton, J. Oprea and D
Tanr\'e. Lusternik--Schnirelmann category. {\em
Mathematical Surveys and Monographs}, \textbf{103}.
American Mathematical Society, Providence, RI, 2003.\

\bibitem{DKR} A. Dranishnikov, M. Katz, and Yu. Rudyak. Small
values of the Lusternik-Schnirelmann category for manifolds, {\em
Geometry and Topology\/} \textbf{12} (2008), 1711-1727.

\bibitem{FHL} Y. Felix, S. Halperin, and J.-M. Lemaire. The rational LS category of products and of Poincar\`e duality complexes. {\em Topology} \textbf{37} (1998) 749-756.

\bibitem{Fe} Y. Felix, S. Halperin, and J.-C. Thomas. {\em Rational Homotopy Theory}, Springer, 2000.

\bibitem{H} K. Hardie. On the category of the double mapping cylinder.
{\em T\^ohoku Math. J.}  (2)  \textbf{25} (1973), 355-358. 

\bibitem{Ha} A. Hatcher. {\em Algebraic Topology}, Cambridge University Press, 2002.

\bibitem{Hi} P. J. Hilton. Suspension Theorems and the generalized Hopf invariant. {\em Proc. of the London Math. Soc.} \textbf{1} (1951), 462-493.

\bibitem{Iw98}N. Iwase. Ganea's conjecture on Lusternik-Schnirelmann category. {\em Bull. Lond. Math. Soc.}, \textbf{30} (1998), 623-634.

\bibitem{Iw02}N. Iwase. A$_\infty$ methods in Lusternik-Schnirelmann category. {\em Topology.} \textbf{41} (2002) 695-723.

\bibitem{LSV} P. Lambrechts, D. Stanley, and L. Vandembroucq. Embeddings up to homotopy of two cones in Euclidean spaces. {\em Trans. American Math. Soc.} \textbf{354} (2002), 3973-4013.

\bibitem{R} J. J. Rivadeneyra-P\'{e}rez.  On $\cat(X\smallsetminus p)$, {\em International Journal of Mathematics and Mathematical Sciences,} \textbf{15} (1992), 812.

\bibitem{Rudbook} Y. Rudyak. {\em On Thom Spectra, Orientability, and Cobordism}. Springer, 2010.

\bibitem{Rudweight} Y. Rudyak. On Category weight and its applications. {\em Topology}, \textbf{50} (2008) 37-55.

\bibitem{Sull} D. Sullivan. {\em Geometric Topology, Localization, Periodicity, and Galois Symmetry}. The MIT Press, 1970.

\bibitem{Sv} A. \v Svarc. The genus of a fibered space. {\em Trudy
Moskov. Mat. Ob\v s\v c} \textbf {10, 11} (1961 and 1962), 217--272, 99--126,
(in {\em Amer. Math. Soc. Transl.} Series 2, vol \textbf{55} (1966)).


\end{thebibliography}
\end{document}